\documentclass[11pt,a4paper]{amsart}
\usepackage{latexsym}
\usepackage{graphicx}
\usepackage{subfig}
\usepackage{caption}
\usepackage{float}
\usepackage{enumerate}
\usepackage[top=3.2cm,bottom=3.8cm,left=3cm,right=2cm]{geometry}
\usepackage{mathrsfs}
\usepackage{amssymb}
\usepackage{amsbsy}
\usepackage{xcolor}
\usepackage[normalem]{ulem}
\usepackage{cancel}
\usepackage[colorlinks,linkcolor=blue,citecolor=blue,pagebackref]{hyperref}

\renewcommand {\thefootnote}{\fnsymbol{footnote}}

\newtheorem{theorem}{Theorem}[section]
\newtheorem{lemma}[theorem]{Lemma}

\newtheorem*{logBMconjecture}{Log Brunn--Minkowski conjecture}
\theoremstyle{definition}

\theoremstyle{remark}
\newtheorem{remark}[theorem]{Remark}

\numberwithin{equation}{section}
\newcommand{\R}{\mathbb{R}}

\newcommand\nnfootnote[1]{%
	\begin{NoHyper}
		\renewcommand\thefootnote{}\footnote{#1}%
		\addtocounter{footnote}{-1}%
	\end{NoHyper}
}

\begin{document}
	\begin{center}
		{\large\bf The uniqueness and non-uniqueness of solutions to the even \\dual Minkowski problem}
	\end{center}
	\vskip 15pt
	\begin{center}
		{\small\bf  Ge\ Xiong\ \  \ \ \ \ \ \ \ Kai-Wen\ Yang}\\~~ \\
		\small{School of Mathematical Sciences, Tongji University, Shanghai, 200092, P. R. China}
	\end{center}
	
	\vskip 5pt
	\nnfootnote{E-mail addresses: 1. xiongge@tongji.edu.cn;\ 2. yangkaiwen@tongji.edu.cn.}
	\nnfootnote{Research of the authors was supported by NSFC No. 12271407.}
	
	\begin{center}
		\begin{minipage}{14cm}
			{{\bf Abstract:} For solutions to the Minkowski problem of the even dual curvature measures $\widetilde{C}_q$ in  \(\R^n\), we prove the non-uniqueness  for \(q>n\)  and $n\ge 2$,  and the uniqueness for \(0<q<n\) and \(n=2\). These results are governed  by our established  logarithmic Brunn-Minkowski inequality for dual quermassintegrals $\widetilde V_q$.} 
			
			
		\end{minipage}
	\end{center}
	
	\vskip 20pt
	\section{\bf Introduction}
	\vskip 5pt
	
	The setting for this paper is the $n$-dimensional Euclidean space $\mathbb{R}^n$, $n\ge 2$.   A \emph{convex body} in $\mathbb{R}^n$ is a compact convex set with non-empty interior. A convex body \(K\) is  \emph{origin-symmetric} if \(K=-K\). In the 1970s, Lutwak \cite{lutwak1} introduced the notion of \emph{dual mixed volume}, which is dual to the classical mixed volume in convex geometry.
	Let \( K_1, \ldots, K_n\) be convex bodies  in \(\mathbb{R}^n \) containing the origin in their interiors, their dual mixed volume is
	\[
	\widetilde{V}(K_1, \ldots, K_n) = \frac{1}{n} \int_{\mathbb{S}^{n-1}} \rho_{K_1}(u) \cdots \rho_{K_n}(u)\, du,
	\]
	where \(\mathbb{S}^{n-1}\) is the \((n-1)\)-dimensional Euclidean unit sphere, and \(\rho_K(x) = \sup\{\lambda \geq 0 : \lambda x \in K\}\), \(x \in \mathbb{R}^n \backslash\{0\}\), is the \emph{radial function} of convex body \( K \).
	
	If \( K_1= \ldots= K_j = K \) and \( K_{j+1}= \ldots= K_n = B^n \), where \(B^n\) is the Euclidean unit ball in \(\R^n\), then 
	\[
	\widetilde{V}(\underbrace{K, \ldots, K}_j, \underbrace{B^n, \ldots, B^n}_{n-j}) = \frac{1}{n} \int_{\mathbb{S}^{n-1}} \rho_K^j(u)\, du, \quad j = 0, 1, \ldots, n,\]
	is called the \emph{\( j \)-th dual quermassintegral} of \( K \) and denoted by $\widetilde{V}_j(K)$. 
	Indeed, the index $j$ involved in $\widetilde{V}_j(K)$ can be extended to $\mathbb{R}$. So, 
	the \( q \)-th dual quermassintegral of  \( K \) is 
	\[
	\widetilde{V}_q(K) = \frac{1}{n} \int_{\mathbb{S}^{n-1}} \rho_K^q(u) \, du, \quad q\in \mathbb{R}.
	\]
	
	In 2016, Huang, Lutwak, Yang and Zhang \cite{huang} defined the geometric measures associated with dual quermassintegrals, called \emph{dual curvature measures}.
	Let $K$ be a convex body in \(\mathbb{R}^n \) with the origin in its interior and $q \in \mathbb{R}$. The $q$-th dual curvature measure $\widetilde{C}_q(K, \cdot)$ is defined by
	\[
	\widetilde{C}_q(K, \eta) = \frac{1}{n} \int_{\alpha_K^*(\eta)} \rho_K^q(u)du, \ \text{for each Borel } \eta \subset \mathbb{S}^{n-1},
	\]
	where $\alpha_K^*(\eta)$ denotes the set of directions $u\in \mathbb{S}^{n-1}$ such that $\rho_K(u)u$ belongs to the \emph{reverse Gauss image} of $\eta$,  $\nu^{-1}_K(\eta)$. Specifically, 
	$\widetilde{C}_0(K, \cdot)$ is  the \emph{radial angle measure} of $K$ and  $\widetilde{C}_n(K, \cdot)$ is the  \emph{cone-volume measure} $V_K$ of $K$.
	
	If, in addition,  $K\in C^2_+$, i.e., $K$ has a $C^2$ boundary with everywhere positive curvature, then
	\[
	\widetilde{C}_q(K, \eta) = \frac{1}{n} \int_{\eta} h_K(u) |\nabla h_K(u)|^{q-n} \det(h_{ij}(u) + h_K(u) \delta_{ij}) \, du,
	\]
	where $h_K(x)=\max\{x \cdot y: y \in K\}$, $x\in \R^n$, is the \emph{support function} of \(K\) and $(h_{ij})$ denotes the Hessian matrix of $h_K$ with respect to an orthonormal frame on $\mathbb{S}^{n-1}$. 
	Therefore, $\widetilde{C}_q(K, \cdot)$ is absolutely continuous with respect to spherical Lebesgue measure with continuous density.
	
    \vskip3pt
	The \emph{dual Minkowski problem} originally posed in \cite{huang} asks: Given a finite Borel measure $\mu$ on $\mathbb{S}^{n-1}$ and $q\in \mathbb{R}$, what are necessary and sufficient conditions for the existence of a convex body $K\subset \mathbb{R}^n$ with the origin in its interior that solves the measure equation
	$\widetilde{C}_q(K, \cdot)=\mu.$
	\vskip3pt
	
	The existence and uniqueness of solutions to the dual Minkowski problem were both solved for $q<0$ by Zhao \cite{zhao}, and up to dilation for $q=0$ by Aleksandrov \cite{AF}. For $0<q\le n$, only the existence of solutions to the problem in the symmetric case has been completely solved: the case \(0<q\le 1\) was proved in \cite{huang}; the case of integer \(q =2,\ldots ,n-1\) was proved by Zhao \cite{Zhao2018}; the case \(1<q<n\) was solved in \cite{huang} with a stronger sufficient condition and later completely solved by B\"or\"oczky, Lutwak, Yang, Zhang and Zhao \cite{B3} under the condition of B\"or\"oczky, Henk and Pollehn \cite{BC}; and the case \(q=n\) was solved by B\"or\"oczky, Lutwak, Yang and Zhang \cite{BL}.
	
	The uniqueness problem in the symmetric case for positive indices remains largely \emph{open}. For $q=n$, it is the even log-Minkowski problem. Stancu treated the discrete planar case \cite{Stancu1,Stancu2}, while B\"or\"oczky, Lutwak, Yang and Zhang determined the full planar case \cite{logBM}. In higher dimensions, Saroglou proved the unconditional case \cite{Saroglou2015}; B\"or\"oczky and Kalantzopoulos proved the case of convex bodies with $n$ independent hyperplane symmetries \cite{BKsym}. Kolesnikov and Milman proved local uniqueness near the Euclidean ball \cite{KM}, whereas Chen, Huang, Li and Liu proved global fixed-reference uniqueness when the reference body is $C^{2,\alpha}_+$ and Hausdorff-close to the ball \cite{CHLL}. Further global fixed-reference uniqueness results under curvature-pinching hypotheses were obtained by Milman \cite{MilmanCentro} and by Ivaki and Milman \cite{IvakiMilman}.
	
	For $0<q<n$, general uniqueness remains open. For constant density, Chen, Huang and Zhao \cite{CHZ} proved that the Euclidean ball is the unique smooth uniformly convex origin-symmetric solution. For even densities $C^\alpha$-close to $1$, Hu \cite{HuPositive} proved existence and uniqueness for $0<q\le n-1$ and $2\le n\le4$, and for $n-4\le q\le n-1$ and $n>4$; B\"or\"oczky, Chen, Liu and Saroglou \cite{BCLSnear} extended this to every $0<q<n$. Hu and Ivaki \cite{HuIvaki2026} proved uniqueness among $C^2_+$ unconditional solutions for every $0<q<n$.
	
	
	On the non-uniqueness of the even dual Minkowski problem, Huang and Jiang \cite{HuangJiang} proved the case where \(n=2\) and \(q\) is an even integer no less than \(6\). Chen, Chen and Li \cite{CCL} later extended the non-uniqueness result to the range \(q>2n\) for all \(n\ge 2\).
	
	The purpose of this article is to further refine the work on this topic. We prove the uniqueness of $\widetilde{C}_q$  for $0<q<2$ and $n=2$, and the non-uniqueness of $\widetilde{C}_q$ for $q>n$ and $n\ge 2$.
	\begin{theorem}\label{unique}
		Let \(K\) and \(L\) be  origin-symmetric convex bodies in \(\R^2\). If  \(\widetilde{C}_q(K,\cdot)=\widetilde{C}_q(L, \cdot)\) for \   \(0<q<2\), then \(K=L\).
	\end{theorem}
	
	\begin{theorem}\label{min}
		If \(q>n\ge 2\), then there exist origin-symmetric convex bodies \(K\) and \(L\) in \(\R^n\) such that
		$\widetilde{C}_q(K, \cdot)=\widetilde{C}_q(L, \cdot),$ while $K\neq L$.
	\end{theorem}
	
	Note that the desired bodies $K$ and $L$ in Theorem \ref{min} can be chosen from the class $C^{\infty}_+$. Consequently, there exists a smooth, positive, even function $f$ on $\mathbb{S}^{n-1}$ such that the Monge-Amp\`ere equation on the sphere $\mathbb{S}^{n-1}$
	\begin{equation*}
		h(u) |\nabla h(u)|^{q-n} \det(h_{ij}(u) + h(u) \delta_{ij}) =f
	\end{equation*}
	has different classical solutions in  $C^\infty_+$ for every \(q>n\). Please refer to Remark \ref{remark} for details. In recent years, 	non-uniqueness of solutions to the Monge-Amp\`ere equations associated with Minkowski problems has  been intensively studied (\cite{CCL,CLZlog,HLW,wang2,liu,Li}).
	
	To prove Theorem \ref{unique} and Theorem \ref{min}, our strategy is to investigate  the \emph{logarithmic Brunn-Minkowski inequality for dual quermassintegrals}. 
	For a continuous  $h:\mathbb{S}^{n-1}\to(0,\infty)$, its \emph{Wulff shape}  is defined by
	\[[h]=\{x\in\mathbb R^n: x\cdot v\le h(v), \text{ for all } v\in\mathbb{S}^{n-1}\}.\]
	The unknown logarithmic Brunn--Minkowski inequality for dual quermassintegrals asserts that
	\begin{equation}\label{logBMq}
		\widetilde V_q((1-\lambda)\cdot K+_0\lambda\cdot L)
		\ge
		\widetilde V_q(K)^{1-\lambda}\widetilde V_q(L)^\lambda,
		\quad 0\le\lambda\le1,
	\end{equation}
	for all origin-symmetric convex bodies $K,L$ in $\mathbb R^n$. 
	Here, $(1-\lambda)\cdot K+_0\lambda\cdot L=[h_K^{1-\lambda}h_L^\lambda]$ is the geometric combination of \(K\) and \(L\).  Since $\widetilde V_n=V_n$, \eqref{logBMq} for $q=n$ becomes the following conjectured  logarithmic Brunn-Minkowski inequality by B\"or\"oczky, Lutwak, Yang and Zhang \cite{logBM}.
	
	\begin{logBMconjecture}
		If $K$ and $L$ are origin-symmetric convex bodies in $\mathbb R^n$, then for all $0<\lambda<1$,
		\[V_n((1-\lambda)\cdot K+_0\lambda\cdot L)
		\ge V_n(K)^{1-\lambda}V_n(L)^\lambda.\]
		In addition, equality is expected to hold if and only if
		\[K=K_1+\cdots+K_m, \quad L=L_1+\cdots+L_m, \]
		for compact convex sets $K_i$ and $L_i, i=1,2, \ldots, m,$ with dimension at least one such that $\sum_{i=1}^m\dim K_i=n$ and $K_i$ and $L_i$ are homothetic for each $i=1,\ldots,m$.
	\end{logBMconjecture}
	The $2$ dimensional logarithmic Brunn-Minkowski inequality was established in \cite{logBM}. In higher dimensions, it is known for complex bodies by Rotem \cite{Rotem}, for unconditional bodies by Saroglou \cite{Saroglou2015}, and for bodies with $n$ independent hyperplane symmetries by B\"or\"oczky and Kalantzopoulos \cite{BKsym}. Local results of the inequality near the Euclidean ball were obtained by Colesanti, Livshyts and Marsiglietti \cite{CLM} and by Kolesnikov and Milman \cite{KM}; Chen, Huang, Li and Liu \cite{CHLL} proved the global inequality for pairs of bodies near the ball. Further fixed-reference log-Minkowski inequalities under curvature-pinching hypotheses were established by Milman \cite{MilmanCentro} and by Ivaki and Milman \cite{IvakiMilman}. Van Handel \cite{vanHandel} proved the local inequality and characterized its equality cases for origin-symmetric zonoids; he also deduced the fixed-reference log-Minkowski inequality for arbitrary origin-symmetric comparison bodies. Xi \cite{XiReverse} gave a new proof of the latter and determined its full equality cases. Wang \cite{WangDeficit} proved a segment-addition monotonicity principle for the local deficit, yielding a new proof of the local zonoid result.
	
In this article, we establish the inequality \eqref{logBMq} with equality conditions for \(0<q<n\) and \(n=2\), and show  that it fails for \(q>n\ge 2\).
	
	\begin{theorem}\label{logmain2}
		If \(K\) and \(L\) are  origin-symmetric convex bodies in \(\R^2\) and \(0<q<2\), then 
		\[\widetilde V_q((1-\lambda)\cdot K+_0\lambda\cdot L) \ge \widetilde V_q(K)^{1-\lambda}\widetilde V_q(L)^\lambda,  \  \forall  \ 0<\lambda<1,\]
		with equality if and only if \(K\) and \(L\) are dilates.
	\end{theorem}
	
	\begin{theorem}\label{logmain}
		If  \(q>n\ge 2\),  then there exist origin-symmetric convex bodies \(K\) and \(L\) in \(\R^n\) such that
		\[
		\widetilde V_q\bigl(\frac{1}{2}\cdot K+_0\frac{1}{2}\cdot L\bigr)
		<
		\widetilde V_q(K)^{\frac{1}{2}}\widetilde V_q(L)^{\frac{1}{2}}.
		\]
	\end{theorem}
	
	
	
	This paper is organized as follows. In Section 2, we collect some basic facts from convex geometry. Theorems \ref{logmain2} and \ref{unique} are proved in Section 3, and Theorems \ref{logmain} and  \ref{min} are proved in Section 4. 
	
	\vskip 20pt
	\section{\bf Preliminaries}
	\vskip 5pt
	
	For quick later reference, we collect some basic facts on convex bodies. Good references  are the books by Gardner \cite{Gardner1} and Schneider \cite{Sch}.
	
	Write $\mathcal K_e^n$ for the class of origin-symmetric convex bodies in $\mathbb{R}^n$. Let $\omega_n$ denote the volume of the unit ball \(B^n \) in $\mathbb{R}^n$. For
    convex bodies 	\(K\), \(L\)  in $\mathbb{R}^n$, their \emph{Minkowski combination}  is $$aK+bL=\{ax+by\,|\, x\in K, y\in L\}, \ \forall \ a,b>0.$$ Then, 
	$h_{aK+bL}(x)=ah_K(x)+bh_L(x),\ x\in \R^n.$
	
	Let \(K\) be a convex body in \(\R^n\) with the origin in its interior. Its \( q \)-th dual quermassintegral 
	\[
	\widetilde{V}_q(K) = \frac{1}{n} \int_{\mathbb{S}^{n-1}} \rho_K^q(u) \, du, \quad q\in \mathbb{R},
	\]
	is obvious  $q$-homogeneous, i.e., $\widetilde{V}_q(rK)=r^q\widetilde{V}_q(K)$ for  $r>0$. By the polar coordinate formula, 
	\begin{equation}\label{VQ}
		\widetilde{V}_q(K) = \frac{q}{n} \int_{K} |x|^{q-n} \, dx, \quad q>0.
	\end{equation}
	
	The dual curvature measure satisfies the following
	\begin{equation*}
		|\widetilde C_q(K,\cdot)|=\widetilde V_q(K),
		\quad
		\widetilde C_q(rK,\cdot)=r^q \widetilde C_q(K,\cdot),\quad r>0.
	\end{equation*}
	
	Let $\Omega\subset\mathbb S^{n-1}$ be a closed set not contained in any closed hemisphere and let $h:\Omega\to(0,\infty)$ be continuous. The \emph{Wulff shape} determined by $h$ is 
	\[
	[h]=\{x\in\mathbb R^n: x\cdot v\le h(v) \text{ for all } v\in\Omega\}.
	\]
	In particular, if $\Omega=\mathbb S^{n-1}$ and $h=h_K$ is the support function of a convex body $K$, then $[h_K]=K$. If $h_i\to h$ uniformly on $\Omega$, then Aleksandrov's convergence theorem for Wulff shapes gives $[h_i]\to [h]$ with respect to the Hausdorff metric.
	
	The \emph{Blaschke selection theorem} reads: If \(\{K_i\}_{i=1}^\infty\) is a uniformly bounded sequence of convex bodies in \(\R^n\), then there exists a subsequence \(\{K_{i_j}\}_{j=1}^\infty\) and a compact convex set \(K\subset\mathbb R^n\) such that
	\(K_{i_j}\to K\)
	with respect to the Hausdorff metric.
	
	The following one-dimensional Pr\'ekopa--Leindler equality is needed. Refer to \cite{GardnerBM} for a good survey on the Pr\'ekopa--Leindler inequality. 
	
	\begin{lemma}[Pr\'ekopa--Leindler]\label{PLlemma}
	Let $0<\lambda<1$. 	If  $f,g,h$ are positive continuous integrable functions on $\mathbb R$  satisfying
		$h((1-\lambda)s+\lambda t)
		\ge f(s)^{1-\lambda}g(t)^\lambda, \ \forall \ s,t\in\mathbb R,$
		then
		\[\int_{\mathbb R}h\ge
		(\int_{\mathbb R}f)^{1-\lambda}
		(\int_{\mathbb R}g)^\lambda.\]
		If the equality holds, then there exist $\tau\in\mathbb R$ and $a>0$ such that for all $s\in\mathbb R$
		\[
		g(s+\tau)=a f(s),
		\qquad
		h(s+\lambda\tau)=f(s)^{1-\lambda}g(s+\tau)^\lambda.
		\]
	\end{lemma}
	
	A point $x\in\partial K$ is called a \emph{smooth boundary point} if there exists a unique exterior unit normal. Denote by $\partial'K$ the set of all smooth boundary points. The spherical image map
	$\nu_K:\partial'K\to\mathbb S^{n-1}$ sends each smooth boundary point to its unique outer unit normal. We shall use the following lemma by B\"or\"oczky and Kalantzopoulos \cite{BKsym}.
	
	\begin{lemma}[{\cite[Lemma 1]{BKsym}}]\label{BKdirectsum}
		Let $K$ be a convex body in $\R^n$, and let $\xi_1,\ldots,\xi_m$, $m\geq2$, be non-trivial complementary linear subspaces which together span $\R^n$. Then
		\[\nu_K(\partial'K)\subset\xi_1\cup\cdots\cup\xi_m\]
		if and only if there exist compact convex sets $K_1,\ldots,K_m$ with
		\(
		\operatorname{lin}(K_i-K_i)=
		\big(\sum_{j\ne i}\xi_j\big)^\perp
		\)
		(and hence $\dim K_i=\dim\xi_i$) for $i=1,\ldots,m$ such that
		$K=K_1+\cdots+K_m.$
		
		If $K$ is unconditional and $K_1,\ldots,K_m$ are unconditional, then $K_i\subset\xi_i$, $i=1,\ldots,m$.
	\end{lemma}
	
	\begin{lemma}[{\cite[Theorem 4.5]{huang}}]
		Let $h_0:\Omega\to(0,\infty)$ and $f:\Omega\to\mathbb R$ be continuous, and let $[h_t]$ be a logarithmic family of Wulff shapes associated with $(h_0,f)$, that is,
		\[
		\log h_t(v)=\log h_0(v)+tf(v)+o(t,v),
		\quad \frac{o(t,\cdot)}{t}\to 0,
		\]
		uniformly on $\Omega$. Then, for $q\ne0$,
		\begin{equation}\label{HLYZ45}
			\frac{d}{dt}\widetilde V_q([h_t])\bigg|_{t=0}
			=q\int_\Omega f(v)\,d\widetilde C_q([h_0],v).
		\end{equation}
	\end{lemma}
	
	\begin{lemma}[{\cite[Lemma 3.6]{huang}}]\label{weak}
		Let \(q\in\mathbb{R}\). If $K_i\to K_0$ in the Hausdorff metric and all $K_i$ (including \(K_0\)) contain the origin in their interiors, then \(\widetilde C_q(K_i,\cdot)\to \widetilde C_q(K_0,\cdot)\), weakly.
	\end{lemma}
	
	Let \( \mu \) be a non-zero finite Borel measure on \( \mathbb{S}^{n-1} \),  $q\ne0$, and let $K$ be a convex body in $\mathbb{R}^n$ with the origin in its interior.
	In \cite{huang}, the authors defined the functional
	\[
	\Phi_\mu(K) = -\frac{1}{|\mu|} \int_{\mathbb{S}^{n-1}} \log h_K(v) \, d\mu(v) + \frac{1}{q}\log \frac{\widetilde{V}_q(K)}{\omega_n},
	\]
	and proved the following variational result.
	\begin{lemma}[{\cite{huang}}]\label{HLYZ-5.1}
		Let \(q>0\) and \(\mu\) be a finite even Borel measure on
		\(\mathbb{S}^{n-1}\) with \(|\mu|>0\).  If \(K\in \mathcal K_e^n\)  
		satisfies that 
		$\Phi_\mu(K)=\sup\{\Phi_\mu(Q):
		\widetilde V_q(Q)=|\mu|,\ Q\in\mathcal K_e^n\},$
		then
		$\widetilde C_q(K,\cdot)=\mu .$
	\end{lemma}
	
The following regularity theorem is the special case  of the result by B\"or\"oczky and Fodor \cite[Theorem 1.5]{BF} with \(p=0\) and \(Q=B^n\).

\begin{theorem}[{\cite{BF}}]\label{BF-regularity}
		Let \(q\in\mathbb R\), \(\alpha\in(0,1)\) and  \(K\) be a convex body in $\R^n$ with the origin in its interior. If
		$d\widetilde C_q(K,\cdot)=f\,du,$	where \(f\in C^\alpha(\mathbb S^{n-1})\) and
		\(0<c_1\le f\le c_2\), then \(K\in C^2_+\) and \(h_K\in C^{2,\alpha}(\mathbb S^{n-1})\). 
	\end{theorem}
To go a step further,	if
\(f\in C^\infty(\mathbb S^{n-1})\), then \(K\in C^\infty_+\).
Indeed, by Theorem \ref{BF-regularity}, it follows that $h_K\in C^{2,\alpha}(\mathbb S^{n-1})$. So the matrix $(h_{ij}+h_K\delta_{ij})$ is positive definite and 
	\[
	\det(h_{ij}+h_K\delta_{ij})
	=nf h_K^{-1}(h_K^2+|\nabla_{\mathbb S^{n-1}}h_K|^2)^{\frac{n-q}{2}}.
	\]
	Since the positive definite matrix  $(h_{ij}+h_K\delta_{ij})$ is  continuous on the compact sphere, its linearized equation is uniformly elliptic. 
	Since $h_K>0$, both $h_K^{-1}$ and $(h_K^2+|\nabla_{\mathbb S^{n-1}}h_K|^2)^{\frac{n-q}{2}}$ are smooth functions of $(h_K,\nabla_{\mathbb S^{n-1}}h_K)$. Thus, if $f\in C^\infty$, the equation is a smooth uniformly elliptic equation in local coordinates. By  \cite[Lemma 17.16]{GT}, it follows that  $h_K\in C^{k+2,\alpha}(\mathbb S^{n-1})$ for every $k\ge1$. Thus $h_K\in C^\infty(\mathbb S^{n-1})$ and therefore $K\in C^\infty_+$.
	
	\vskip 10pt
	\section{\bf The uniqueness}
	\vskip 5pt
	
	\begin{lemma}\label{layercake}
		If $A\in\mathcal K_e^n$ and $0<q<n$, then
		\[
		\widetilde{V}_q(A)=\frac{q(n-q)}n\int_0^\infty r^{q-n-1}V_n(A\cap rB^n)\,dr.\]
	\end{lemma}
	
	\begin{proof}
Since
		\[
		|x|^{q-n}=(n-q)\int_{|x|}^\infty r^{q-n-1}\,dr, \ \forall x\ne0,
		\]
from \eqref{VQ} and the Fubini theorem, we have
		\[
		\begin{aligned}
			\widetilde{V}_q(A)
			&=\frac qn\int_A |x|^{q-n}\,dx 
			=\frac{q(n-q)}n\int_A\int_{|x|}^\infty r^{q-n-1}\,dr\,dx \\
			&=\frac{q(n-q)}n\int_0^\infty r^{q-n-1}V_n(A\cap rB^n)\,dr.
		\end{aligned}
		\]
		Since $A\cap rB^n=rB^n$ for all sufficiently small $r>0$ and $V_n(A\cap rB^n)\le V_n(A)$  for all $r>0$, the integral is finite. 
	\end{proof}
	
	The following lemma  will be used in the equality discussion. Here a \emph{non-trivial direct-sum decomposition} means a decomposition $K=K_1+\cdots+K_m$, $m\ge2$, into compact convex sets of positive dimensions whose direction spaces $\operatorname{lin}(K_i-K_i)$ are complementary.
	
	\begin{lemma}\label{sphericalpatch}
		Let $K$ be a convex body in $\R^n$ with the origin in its interior, and let
		\[
		r_-(K)=\min_{u\in\mathbb S^{n-1}}\rho_K(u),
		\quad
		r_+(K)=\max_{u\in\mathbb S^{n-1}}\rho_K(u).
		\]
		If $r\in(r_-(K),r_+(K))$, then $K\cap rB^n$ admits no non-trivial direct-sum decomposition. 
	\end{lemma}
	\begin{proof}
		Let $r\in(r_-(K),r_+(K))$ and 
		\[
		U=\{u\in\mathbb S^{n-1}:\rho_K(u)>r\}.
		\]
		Since \(\rho_K\) is continuous and \(\mathbb S^{n-1}\) is compact, $U$ contains a non-empty relatively open subset of $\mathbb S^{n-1}$. For $u\in U$, the point $ru$ belongs to $\operatorname{int}K$ and lies on $\partial(rB^n)$. Hence the set $rU$ is contained in $\partial(K\cap rB^n)$ and is smooth, and every boundary point \(ru\in rU\) has outer unit normal $u$, and therefore $U\subset \nu_{K\cap rB^n}(\partial'(K\cap rB^n))$.
		
		Assume that $K\cap rB^n$ admits a non-trivial direct-sum decomposition. By Lemma \ref{BKdirectsum}, $\nu_{K\cap rB^n}(\partial'(K\cap rB^n))$ is contained in a finite union of proper linear subspaces of $\R^n$. But  $U\subset \nu_{K\cap rB^n}(\partial'(K\cap rB^n))$, which is impossible since a non-empty relatively open subset of $\mathbb S^{n-1}$ cannot be contained in a finite union of proper linear subspaces.
	\end{proof}

	\begin{lemma}\label{logToDualOne}
		Let $0<q<n$ and $0<\lambda<1$. If the log Brunn--Minkowski inequality
		\begin{equation}\label{volumeLBM}
			V_n([h_K^{1-\lambda}h_L^\lambda])
			\ge
			V_n(K)^{1-\lambda}V_n(L)^\lambda
		\end{equation}
		holds for all $K,L\in\mathcal K_e^n$, then 
		\begin{equation}\label{dualLBM}
			\widetilde V_q([h_K^{1-\lambda}h_L^\lambda])
			\ge
			\widetilde V_q(K)^{1-\lambda}\widetilde V_q(L)^\lambda.
		\end{equation}
		If the equality in \eqref{volumeLBM} is given by the direct-sum equality condition stated in the log Brunn--Minkowski conjecture, then the equality in \eqref{dualLBM} holds if and only if $K$ and $L$ are dilates.
	\end{lemma}
	
	\begin{proof}We divide the proof into two steps.
		
		\textbf{Step 1.} Prove the inequality \eqref{dualLBM}.
		
		Write
		\(M=[h_K^{1-\lambda}h_L^\lambda].\)
		For $a,b>0$, let
		\[
		K_a=K\cap aB^n,
		\quad
		L_b=L\cap bB^n,
		\quad
		N_{a,b}=[h_{K_a}^{1-\lambda}h_{L_b}^\lambda].
		\]
		Since $h_{K_a}\le h_K$ and $h_{L_b}\le h_L$, it gives that $N_{a,b}\subset M$; Since $h_{K_a}\le a$ and $h_{L_b}\le b$, we  have $N_{a,b}\subset a^{1-\lambda}b^\lambda B^n$. Hence,
		$N_{a,b}\subset M\cap a^{1-\lambda}b^\lambda B^n.$
		
		By the monotonicity of the volume functional \(V_n\) and \eqref{volumeLBM} for $K_a$ and $L_b$, we obtain
		\begin{equation}\label{Nab}
			V_n(M\cap a^{1-\lambda}b^\lambda B^n)
			\ge V_n(N_{a,b})
			\ge V_n(K_a)^{1-\lambda}V_n(L_b)^\lambda.
		\end{equation}
		Let $a=e^s,$  $b=e^t$, and 
		$G_A(s)=e^{(q-n)s}V_n(A\cap e^sB^n),\ A\in \mathcal{K}_e^n.$
		Then
		\[
		G_M((1-\lambda)s+\lambda t)
		\ge
		G_K(s)^{1-\lambda}G_L(t)^\lambda, \ \text{for} \ 0<\lambda<1.
		\]
		By Lemma \ref{PLlemma}, we have
		\[
		\int_{\mathbb R}G_M(s)\,ds
		\ge
		(\int_{\mathbb R}G_K(s)\,ds)^{1-\lambda}
		(\int_{\mathbb R}G_L(s)\,ds)^\lambda, \ \text{for} \ 0<\lambda<1.
		\]
		Added with 
		\[
		\int_{\mathbb R}G_A(s)\,ds
		=
		\int_0^\infty r^{q-n-1}V_n(A\cap rB^n)\,dr
		\]
		and	Lemma \ref{layercake}, we obtain  \eqref{dualLBM}.

		\textbf{Step 2.} Handle the equality condition.
		
If  $K$ and $L$ are dilates,	by the homogeneity of $\widetilde V_q$, 	the equality holds.
	
    Assume the equality holds in \eqref{dualLBM}, i.e.
		\[
		\int_{\mathbb R}G_M(s)\,ds
		=
		\big(\int_{\mathbb R}G_K(s)\,ds\big)^{1-\lambda}
		\big(\int_{\mathbb R}G_L(s)\,ds\big)^\lambda.
		\]
	By Lemma \ref{PLlemma}, there exist $\tau\in\mathbb R$ and $c>0$ such that for all $s\in\mathbb R$,
		\[G_L(s+\tau)=cG_K(s)
		\quad\text{and}\quad
		G_M(s+\lambda\tau)=G_K(s)^{1-\lambda}G_L(s+\tau)^\lambda.
		\]
		 Letting $s\to-\infty$ in the first identity gives $c=e^{q\tau}$. So
		\begin{equation}\label{vol}
			V_n(L\cap e^{s+\tau}B^n)=e^{n\tau}V_n(K\cap e^sB^n),
			\quad s\in\mathbb R,
		\end{equation}
	and the second identity becomes  the following
		\[
		V_n(M\cap e^{s+\lambda\tau} B^n)
		=
		V_n(K\cap e^sB^n)^{1-\lambda}V_n(L\cap e^{s+\tau}B^n)^\lambda.
		\]

Meanwhile, 		by \eqref{Nab} we have
		\[
		V_n(M \cap e^{s+\lambda \tau} B^n) \geq V_n(N_{e^s, e^{s + \tau}}) \geq V_n(K \cap e^s B^n)^{1-\lambda} V_n(L \cap e^{s+\tau} B^n)^\lambda.
		\]
		Hence,
		\[
		V_n(N_{e^s,e^{s+\tau}})
		= V_n(K\cap e^sB^n)^{1-\lambda}V_n(L\cap e^{s+\tau}B^n)^\lambda
		\]
		and the pair
		$K\cap e^sB^n$
		and $L\cap e^{s+\tau}B^n$
		enjoy the equality condition in \eqref{volumeLBM} for every $s\in\mathbb R$.
		
		If $K$ is a Euclidean ball, then for all sufficiently large $s$ the two truncated bodies are $K$ and $L$. Since  $\partial K$ is smooth and $\nu_K(\partial'K)=\mathbb S^{n-1}$, by Lemma \ref{BKdirectsum}, $K$ admits no non-trivial direct-sum decomposition.
		The assumption on equality conditions in \eqref{volumeLBM} implies that $L$ is a dilate of $K$. 
		
		Assume $K$ is not a Euclidean ball. Let
		\[
		r_-(K)=\min_{u\in\mathbb S^{n-1}}\rho_K(u),
		\quad
		r_+(K)=\max_{u\in\mathbb S^{n-1}}\rho_K(u).
		\]
		For every $r\in(r_-(K),r_+(K))$, Lemma \ref{sphericalpatch} shows that $K\cap rB^n$ admits no non-trivial direct-sum decomposition. By the  equality conditions in \eqref{volumeLBM}, 
		$L\cap e^\tau rB^n=c_r(K\cap rB^n)$
		for some $c_r>0$. \eqref{vol} gives $c_r=e^\tau$, and therefore
		\[
		L\cap e^\tau rB^n=e^\tau(K\cap rB^n),
		\quad r\in(r_-(K),r_+(K)).
		\]
		It follows that $e^\tau K\subset L$. Indeed, if \(x\in K\) and \(|x|<r_+(K)\), choose \(r\in(r_-(K),r_+(K))\) with \(r>|x|\), then \(e^\tau x\in L\); if \(|x|=r_+(K)\), then \(e^\tau x\in L\) follows by approximation. Letting $s\to\infty$ in the volume identity gives $V_n(L)=e^{n\tau}V_n(K)$. Thus $L=e^\tau K$.
	\end{proof}
	
	\begin{proof}[Proof of Theorem \ref{logmain2}]
		Since the planar log Brunn--Minkowski conjecture has been  proved in \cite{logBM}, by Lemma \ref{logToDualOne}, we conclude the proof of Theorem \ref{logmain2} immediately.
	\end{proof}
	
	\begin{proof}[Proof of Theorem \ref{unique}]
		We divide the proof into two steps.
		
		\textbf{Step 1.} Prove the log Minkowski inequality for \(\widetilde V_q\).
		
Let \(K,L\in\mathcal K_e^2\).	For \(q>0\), write
		\[
		\overline{V_q}(K)=\big(\frac{\widetilde V_q(K)}{\omega_2}\big)^{\frac1q}, \quad \text{and} \quad K_t=[h_K^{1-t}h_L^t],\quad \forall \ 0\le t\le1.
		\]
		
		Let
		\[
		\varphi(t)=\log\overline{V_q}(K_t), \quad 0\le t\le1.
		\]
		If \(r=(1-\theta)s+\theta t\), then
		\[
		h_{K_s}^{1-\theta}h_{K_t}^{\theta}\le h_K^{1-r}h_L^r.
		\]
		By the  monotonicity of \(\widetilde{V}_q\) and Theorem \ref{logmain2}, we have 
		\[
		\begin{aligned}
			&\overline{V_q}(K_r)=\overline{V_q}([h_K^{1-r}h_L^r])
			\ge \overline{V_q}([h_{K_s}^{1-\theta}h_{K_t}^{\theta}])
			\ge	\overline{V_q}([h_{K_s}])^{1-\theta}\overline{V_q}([h_{K_t}])^\theta =
			\overline{V_q}(K_s)^{1-\theta}\overline{V_q}(K_t)^\theta,
		\end{aligned}
		\]
		which implies that  \(\varphi(t)\) is concave on $[0,1]$.
		
	Since \(\varphi\) is concave,	by the variational formula \eqref{HLYZ45}, it follows that 
		\[\varphi'_+(0)=
		\frac1{\widetilde V_q(K)}
		\int_{\mathbb S^1}\log\frac{h_L}{h_K}\,d\widetilde C_q(K,\cdot)\ge \varphi(1)-\varphi(0).
		\]
		Thus
		\begin{equation}\label{logMineq}
			\log\frac{\overline{V_q}(L)}{\overline{V_q}(K)}
			\le
			\frac1{\widetilde V_q(K)}
			\int_{\mathbb S^1}\log\frac{h_L}{h_K}\,d\widetilde C_q(K,\cdot).
		\end{equation}
		
		If the equality holds in \eqref{logMineq}, then \(\varphi'_+(0)=\varphi(1)-\varphi(0)\). So the concave function \(\varphi\) is affine on \([0,1]\), and therefore the equality holds in the log Brunn--Minkowski inequality for \(\widetilde V_q\) for  \(K,L\). Thus,  \(K\) and \(L\) are dilates. 
       
 Conversely,  if \(K\) and \(L\) are dilates, it immediately  gives the equality in \eqref{logMineq}.
		
		\textbf{Step 2.} Prove the uniqueness.
		
		Suppose that \(K,L\in\mathcal K_e^2\) satisfying 
		$\widetilde C_q(K,\cdot)=\widetilde C_q(L,\cdot)=\mu.$
		
		Since \(|\widetilde C_q(M,\cdot)|=\widetilde V_q(M)\), we have
		$\widetilde V_q(K)=|\mu|=\widetilde V_q(L).$
		Applying \eqref{logMineq} to \((K,L)\), we obtain 
		\[
		0\le
		\frac1{|\mu|}\int_{\mathbb S^1}\log\frac{h_L}{h_K}\,d\mu.
		\]
		Applying \eqref{logMineq} to \((L,K)\) gives 
		\[
		0\le \frac1{|\mu|}\int_{\mathbb S^1}\log\frac{h_K}{h_L}\,d\mu =-
		\frac1{|\mu|}\int_{\mathbb S^1}\log\frac{h_L}{h_K}\,d\mu.
		\] So both equalities hold, and \(L=cK\) for some \(c>0\). Since
		\[
		\widetilde C_q(cK,\cdot)=c^q\widetilde C_q(K,\cdot)
		\]
		and \(q>0\), we get \(c=1\). Thus \(K=L\).
	\end{proof}
	
	
	
	\vskip 10pt
	\section{\bf The non-uniqueness}
	\vskip 5pt
	
	\begin{proof}[Proof of Theorem \ref{logmain}]
		We divide the proof into two steps.
		
		\noindent
		\textbf{Step 1: The two-dimensional case.}
		
		Assume \(q>2\). For \(t\in(-1,1)\), define a family of rectangles by
		\[
		A_t=[-(1+t),1+t]\times[-(1-t),1-t]\subset\mathbb R^2.
		\]
		Then
		\[
		\widetilde V_{q}(A_t)=\widetilde V_{q}(A_{-t}),\quad \text{and} \quad \frac{A_t+A_{-t}}2=A_0=[-1,1]^2.
		\]
		
		Let $r_t=(1-t^2)^{\frac{1}{2}}$.  For $u=(u_1,u_2)\in \mathbb S^{1}$, set $a=|u_1|$ and $b=|u_2|$. Then
		\[
		h_{A_t}(u)=(1+t)a+(1-t)b,
		\quad
		h_{A_{-t}}(u)=(1-t)a+(1+t)b
		\]
        and
		\[
		(h_{A_t}h_{A_{-t}})(u)=(a+b)^2-t^2(a-b)^2\ge r_t^2(a+b)^2=h_{r_tA_0}(u)^2.
		\]
		So, $[(h_{A_t}h_{A_{-t}})^{\frac{1}{2}}]\supset r_tA_0$. 
        
        Meanwhile, for $i=1,2,$ since 
        $(h_{A_t}h_{A_{-t}})^{\frac12}(\pm e_i)=r_t,$ by the definition of Wulff shape it follows that 
		$x\cdot(\pm e_i)\le r_t \ \text{for} \ x\in[(h_{A_t}h_{A_{-t}})^{\frac12}]$. So, 
		$[(h_{A_t}h_{A_{-t}})^{\frac12}]\subset[-r_t,r_t]^2=r_tA_0$. 
        
        Therefore, $
		[(h_{A_t}h_{A_{-t}})^{\frac{1}{2}}]=r_tA_0.$
		
		From formula \eqref{VQ}, we have
		\[\widetilde V_q(A_t)=
		2q\int_0^{1+t}\int_0^{1-t}
		(x^2+y^2)^{\frac{q-2}{2}}\,dy\,dx.\]
		So
		\[
		\begin{aligned}
			\frac{d}{dt}\widetilde V_q(A_t)
			=2q(\int_0^{1-t}((1+t)^2+y^2)^{\frac{q-2}{2}}\,dy
			-\int_0^{1+t}(x^2+(1-t)^2)^{\frac{q-2}{2}}\,dx)
		\end{aligned}\]
		and
		\[
		\begin{aligned}
			\frac{d^2}{dt^2}\widetilde V_q(A_t)
			= &
			2q\big(
			(q-2)(1+t)
			\int_0^{1-t}
			((1+t)^2+y^2)^{\frac{q-4}{2}}\,dy  \\
			&+
			(q-2)(1-t)
			\int_0^{1+t}
			(x^2+(1-t)^2)^{\frac{q-4}{2}}\,dx  -
			2((1+t)^2+(1-t)^2)^{\frac{q-2}{2}}
			\big).
		\end{aligned}
		\]
		Hence,
\begin{equation}\label{logFprimeFsecond}
			\begin{aligned}
				\frac{d}{dt}\widetilde V_q(A_t)\Big|_{t=0}=0
				\quad \text{and} \quad 
				\frac{d^2}{dt^2}\widetilde V_q(A_t)\Big|_{t=0}
				=4q((q-2)\int_0^1(1+s^2)^{\frac{q-4}{2}}\,ds
				-2^{\frac{q-2}{2}}).
			\end{aligned}
		\end{equation}
		
Let
\[F_q(t)=\widetilde V_q(A_t)-r_t^q\widetilde V_q(A_0), \quad \forall \ t \in (-1,1).
		\]
Since $
			\widetilde V_q(A_0)=4\int_0^1(1+s^2)^{\frac{q-2}{2}}\,ds,$	$\frac{d}{dt}r_t^q\big|_{t=0}=0,$ and $\frac{d^2}{dt^2}r_t^q\big|_{t=0}=-q$, by \eqref{logFprimeFsecond}, we have
		\[
		F_q'(0)=0, \quad \text{and}\quad F_q''(0)=4q(q-2)\int_0^1(1-s^2)(1+s^2)^{\frac{q-4}{2}}\,ds>0.
		\]
		Indeed,  from the  identity
		\[
		2^{\frac{q-2}{2}}
		=\int_0^1 \frac{d}{ds}\big(s(1+s^2)^{\frac{q-2}{2}}\big)\,ds
		=\int_0^1(1+s^2)^{\frac{q-2}{2}}\,ds
		+(q-2)\int_0^1s^2(1+s^2)^{\frac{q-4}{2}}\,ds,
		\]
		we have
		\[
        \begin{aligned}
            F_q''(0)&= \frac{d}{dt}\widetilde V_q(A_t)\Big|_{t=0} +q \widetilde  V_q(A_0) 
            \\&=4q\big((q-2)\int_0^1(1+s^2)^{\frac{q-4}{2}}\,ds-2^{\frac{q-2}{2}}\big)+4q\int_0^1(1+s^2)^{\frac{q-2}{2}}\,ds
            \\
            &=4q(q-2)\int_0^1(1-s^2)(1+s^2)^{\frac{q-4}{2}}\,ds>0.
        \end{aligned}
		\]

		Consequently, for sufficiently small $t\ne0$,
		\[
		\widetilde V_q(A_t)^{\frac{1}{2}}\widetilde V_q(A_{-t})^{\frac{1}{2}}= \widetilde V_q(A_t)>r_t^q\widetilde V_q(A_0)=\widetilde V_q(r_tA_0)=\widetilde V_q\big([(h_{A_t}h_{A_{-t}})^{\frac{1}{2}}]\big).
		\]
		
		\noindent
		\textbf{Step 2: The case $n\ge3$.}
		
		Assume  \(q>n\). For \(x\in \mathbb R^n\), write
		$
		x=(u,v)\in\mathbb R^2\times\mathbb R^{n-2}.$
		Since \(q-n+2>2\), by the arguments in Step 1, for sufficiently small \(t\ne0\), we have
		\begin{equation}\label{logthinmainnew}
			\int_{A_t}|u|^{q-n}\,du
			>
			r_t^{q-n+2}\int_{A_0}|u|^{q-n}\,du.
		\end{equation}
		
		For \(\varepsilon>0\) and \(t\in(-1,1)\), define a family of parallelotopes in $\mathbb{R}^n$
		\[
		K_{t,\varepsilon}
		=
		A_t\times[-\varepsilon,\varepsilon]^{n-2}
		\subset\mathbb R^n.
		\]
		Then
		\[
		\widetilde V_{q}(K_{t,\varepsilon})
		=
		\widetilde V_{q}(K_{-t,\varepsilon}), \quad \text{and} \quad \frac{K_{t,\varepsilon}+K_{-t,\varepsilon}}2
		=
		A_0\times[-\varepsilon,\varepsilon]^{n-2}
		=
		K_{0,\varepsilon}.
		\]
		
		For \(t\in(-1,1)\), define
		\[
		S_{t,\varepsilon}=r_tA_0\times[-\varepsilon,\varepsilon]^{n-2}.
		\]
		For $z=(z_1,z_2,\ldots,z_n)\in\mathbb R^n$, let
		\[
		a=|z_1|,
		\quad b=|z_2|,
		\quad c=\varepsilon\sum_{i=3}^n|z_i|.
		\]
		Then
		\[
		h_{K_{t,\varepsilon}}(z)=(1+t)a+(1-t)b+c,
		\quad
		h_{S_{t,\varepsilon}}(z)=r_t(a+b)+c,
		\]
	and
		\[
		h_{K_{t,\varepsilon}}(z)h_{K_{-t,\varepsilon}}(z)-h_{S_{t,\varepsilon}}(z)^2
		=4t^2ab+2(1-r_t)c(a+b)\ge0.
		\]
		So, $[(h_{K_{t,\varepsilon}}h_{K_{-t,\varepsilon}})^{\frac{1}{2}}]\supset S_{t,\varepsilon}$. Since the equality holds at all coordinate directions $\pm e_i$, $i=1,\ldots,n$, every point of $[(h_{K_{t,\varepsilon}}h_{K_{-t,\varepsilon}})^{\frac{1}{2}}]$ lies in $S_{t,\varepsilon}$. Hence,
		$[(h_{K_{t,\varepsilon}}h_{K_{-t,\varepsilon}})^{\frac{1}{2}}]=S_{t,\varepsilon}.$

		By formula \eqref{VQ},
		\[
		\begin{aligned}
			\widetilde V_q(K_{t,\varepsilon})
			&=\frac qn\int_{A_t}\int_{[-\varepsilon,\varepsilon]^{n-2}}(|u|^2+|v|^2)^{\frac{q-n}{2}}\,dv\,du\\
			&=\frac qn\varepsilon^{n-2}\int_{A_t}\int_{[-1,1]^{n-2}}(|u|^2+\varepsilon^2|w|^2)^{\frac{q-n}{2}}\,dw\,du.
		\end{aligned}
		\]
		The dominated convergence theorem gives that for \(t\in (-1,1)\),
		\[
		\widetilde V_q(K_{t,\varepsilon})
		=
		\frac{q}{n}(2\varepsilon)^{n-2}\int_{A_t}|u|^{q-n}\,du
		+
		o(\varepsilon^{n-2}), \quad \text{as}\ \varepsilon\to 0.
		\]
		
		Similarly, we have
		\begin{equation*}
			\widetilde V_q(S_{t,\varepsilon})
			=
			\frac qn(2\varepsilon)^{n-2}r_t^{q-n+2}\int_{A_0}|u|^{q-n}\,du+o(\varepsilon^{n-2}),
			\quad \text{as}\ \varepsilon\to0.
		\end{equation*}
		Therefore, by \eqref{logthinmainnew}, for sufficiently small $\varepsilon>0$ we have
		\begin{equation*}
			\widetilde V_q(K_{t,\varepsilon})^{\frac{1}{2}}\widetilde V_q(K_{-t,\varepsilon})^{\frac{1}{2}}=\widetilde V_q(K_{t,\varepsilon})>
			\widetilde V_q(S_{t,\varepsilon})=\widetilde V_q\big([(h_{K_{t,\varepsilon}}h_{K_{-t,\varepsilon}})^{\frac{1}{2}}]\big).
		\end{equation*}
		This completes the proof.
	\end{proof}
	
	To prove Theorem \ref{min}, the following  lemma is needed.
	
	\begin{lemma}\label{smooth-existence}
		Let \(q>n\),  \(S\in\mathcal K_e^n\cap C^2_+\) and
		\(
		\mu=\widetilde C_q(S,\cdot).
		\)
		Then the functional
		\[
		\Phi_\mu(Q)
		=
		-\frac1{|\mu|}\int_{\mathbb{S}^{n-1}}\log h_Q\,d\mu+\frac{1}{q}\log \frac{\widetilde{V}_q(Q)}{\omega_n}
		\]
		attains its maximum on \(\{Q:
		\widetilde V_q(Q)=|\mu|,\ Q\in\mathcal K_e^n\}\).
	\end{lemma}

	\begin{proof}
		We divide the proof into three steps.
		
		\noindent
		\textbf{Step 1:} Prove that the functional has a uniform bound.
		
		Since \(S\in C^2_+\), the measure \(\mu=\widetilde C_q(S,\cdot)\) has a positive
		continuous density
		\[
		d\mu(u)	= \frac1n h_S(u)
		|\nabla h_S(u)|^{q-n}
		\det(h_{ij}(u)+h_S(u)\delta_{ij})\,du .
		\]
		So,
		\(d\mu\le C\,du\)
		for some \(C<\infty\) and
		\[
		\gamma_\mu:=
		\inf_{e\in \mathbb{S}^{n-1}}
		\int_{\mathbb{S}^{n-1}}\log|\langle u,e\rangle|\,d\mu(u)\ge C\int_{\mathbb{S}^{n-1}}\log|\langle u,e\rangle|\,du  >-\infty .
		\]
	
		Let
		\[
		J(Q)=\int_{\mathbb{S}^{n-1}}\log h_Q\,d\mu .	\]
		Then it suffices to minimize the functional  $J$ over  the set \(\{Q:
		\widetilde V_q(Q)=|\mu|,\ Q\in\mathcal K_e^n\}\). 
		
		Fix \( Q\in\mathcal K_e^n\) with \(\widetilde V_q(Q)=|\mu|\). Let
		$R(Q)=\max_{x\in Q}|x|.$
		We have \(\rho_Q(u)\le R(Q)\) for all $ u \in \mathbb{S}^{n-1}$. Hence,
		\[|\mu|=\widetilde V_q(Q)=
		\frac1n\int_{\mathbb{S}^{n-1}}\rho_Q(u)^q\,du
		\le
		\omega_n R(Q)^q,\]
		and therefore
		\[
		R(Q)\ge r_0 := (\frac{|\mu|}{\omega_n})^{\frac{1}{q}}.
		\]
		
		Choose \(x=R(Q)e\in Q\), \(e\in \mathbb{S}^{n-1}\). Since $Q$ is origin-symmetric, it follows  that
		for all  $u\in \mathbb{S}^{n-1}$,
		\(
		h_Q(u)\ge R(Q)|\langle u,e\rangle|.
		\)
		Hence, 
		\[J(Q)=\int_{\mathbb{S}^{n-1}}\log h_Q\,d\mu
		\ge
		|\mu|\log R(Q)+\gamma_\mu
		\ge
		|\mu|\log r_0+\gamma_\mu,\]
	and	therefore, \[\inf_{\{Q:
			\widetilde V_q(Q)=|\mu|,\ Q\in\mathcal K_e^n\}} J(Q)>-\infty.\]
		
		\noindent
		\textbf{Step 2:} Show that a minimizing sequence of $J$ is uniformly bounded.
		
		Let \(\{Q_j\} \subset\mathcal K_e^n\) be a minimizing sequence of $J$. Since for each $j,$
 \[ J(Q_j)\ge |\mu|\log R(Q_j)+\gamma_\mu, \ \] for all $j$ we have 
		\[ \log R(Q_j)\le \frac{J(Q_j)-\gamma_\mu}{|\mu|}\le \frac{M-\gamma_\mu}{|\mu|},\]
		which implies that \(\{Q_j\}\) is uniformly bounded.
		
		\noindent
		\textbf{Step 3:} Prove the existence of the maximizer of $\Phi_\mu$.
		
		By the Blaschke selection theorem, there  exists an origin-symmetric compact convex set $Q_0$ and, after passing to a subsequence, still denoted by \(\{Q_j\}\), such that
		$Q_{j}\to Q_0$ as $j \to +\infty,$
		with respect to the Hausdorff metric.
		
		We claim that \(Q_0\) has non-empty interior. Otherwise, \(V_n(Q_0)=0\), and the continuity of volume under Hausdorff convergence gives \(V_n(Q_j)\to0\). Choose \(R>0\) such that \(Q_j\subset RB^n\) for all \(j\). Since \(q>n\), \eqref{VQ} gives
		\[
		\widetilde V_q(Q_j)
		=\frac qn\int_{Q_j}|x|^{q-n}\,dx
		\le \frac qn R^{q-n}V_n(Q_j)\to0,
		\]
		which contradicts that $\widetilde V_q(Q_{j})=|\mu|$. Thus \(Q_0\in\mathcal K_e^n\).
		
		Since \(h_{Q_{j}}\to h_{Q_0}\) uniformly, \(h_{Q_0}>0\), and \(\widetilde V_q\) is
		continuous, it follows that
		\[
		\widetilde V_q(Q_0)=|\mu|
		\quad \text{and} \quad
		J(Q_0)=\lim_{j\to +\infty}J(Q_{j}).
		\]
		Thus \(Q_0\) minimizes \(J\), and equivalently \(Q_0\) maximizes
		\(\Phi_\mu\).
	\end{proof}
	
	\begin{proof}[Proof of Theorem \ref{min}]
		We prove the following slightly stronger statement: for every \(q>n\),
		there exist \(K,L\in\mathcal K_e^n\) such that
		\(
		\widetilde C_q(K,\cdot)=\widetilde C_q(L,\cdot),
		\) but
		\[
		\Phi_{\widetilde C_q(K,\cdot)}(K)\ne
		\Phi_{\widetilde C_q(K,\cdot)}(L).
		\]
		This immediately yields that $K\neq L$. 
		
		We argue by contradiction. Assume there exists \(q>n\) such that whenever \(K,L\in\mathcal K_e^n\) satisfy
		\(
		\widetilde C_q(K,\cdot)=\widetilde C_q(L,\cdot),
		\)
		one has $\Phi_{\widetilde{C}_q(K, \cdot)}(K)= \Phi_{\widetilde{C}_q(K, \cdot)}(L).$ If so, 
		this will result in the log Brunn--Minkowski inequality
		\[
		\widetilde V_q\big([(h_Kh_L)^{\frac{1}{2}}]\big)
		\ge
		\widetilde V_q(K)^{\frac{1}{2}}\widetilde V_q(L)^{\frac{1}{2}}
		\] for all \(K,L\in\mathcal K_e^n\), and contradicts Theorem \ref{logmain}.
		
		We divide the proof into two steps.
		
		\noindent
		\textbf{Step 1:} Prove the dual log Minkowski inequality under the assumption.
		
		Let \(q>n\), \(S\in\mathcal K_e^n\cap C^2_+\), \(K\in\mathcal K_e^n\), and
		\(\mu=\widetilde C_q(S,\cdot).\) By Lemma~\ref{smooth-existence}, the functional
		\(\Phi_\mu\) attains its maximum on
		$\{Q:\widetilde V_q(Q)=|\mu|,\ Q\in\mathcal K_e^n\}.$
		Let \(Q_0\) be such a maximizer. By Lemma \ref{HLYZ-5.1}, we have
		$\widetilde C_q(Q_0,\cdot)=\mu=\widetilde C_q(S,\cdot).$
		
		By the definition of \(\mu\) and the assumption on  uniqueness of the value of \(\Phi\), we have
		\[
		\Phi_\mu(Q_0)=\Phi_\mu(S).
		\]
		By the maximality of \(Q_0\) for \(\Phi_\mu\), for all \(Q\in\mathcal K_e^n\) with \(\widetilde V_q(Q)=|\mu|\), 
		\[
		\Phi_\mu(Q)\le \Phi_\mu(Q_0) = \Phi_\mu(S).
		\]
		
		Set \(\lambda=(\frac{|\mu|}{\widetilde V_q(K)})^{\frac{1}{q}}\). Then, $\widetilde V_q(\lambda K)=\lambda^q \widetilde V_q(K) =|\mu|$ and \(\lambda K\in \mathcal{K}^n_e\).
		Taking \(Q=\lambda K\) in the above inequality, we have
		\[
		-\frac1{|\mu|}\int_{\mathbb{S}^{n-1}}\log (\lambda h_K)\,d\mu
		+
		\frac1q\log\frac{\lambda^q\widetilde V_q(K)}{\omega_n}
		\le
		-\frac1{|\mu|}\int_{\mathbb{S}^{n-1}}\log h_S\,d\mu
		+
		\frac1q\log\frac{\widetilde V_q(S)}{\omega_n}.
		\]
		That is,
		\begin{equation}\label{M1}
			\log\frac{\overline{V_q}(K)}{\overline{V_q}(S)}
			\le
			\frac1{|\mu|}
			\int_{\mathbb{S}^{n-1}}\log\frac{h_K}{h_S}\,d\mu,
		\end{equation}
		where
		\[
		\overline{V_q}(M)=(\frac{\widetilde V_q(M)}{\omega_n})^{\frac{1}{q}},\quad \text{for all} \ M \in\mathcal K_e^n.
		\]
		
		Since $\mathcal K_e^n\cap C^2_+$ is dense in $\mathcal K_e^n$, the weak continuity of $\widetilde C_q$ (Lemma \ref{weak}) and the continuity of $\widetilde V_q$ allow us to pass to the limit in \eqref{M1}. Since \(S\) and \(K\) are arbitrary, \eqref{M1} holds for every $K, S\in\mathcal K_e^n$, with $\mu=\widetilde C_q(S,\cdot)$.
		
		\noindent
		\textbf{Step 2:} From the dual log Minkowski inequality \eqref{M1} to \eqref{logBMq}.
		
		Let $K,L\in\mathcal K_e^n$. For $0<\lambda<1$, let 
		\[
		g=h_K^{1-\lambda}h_L^\lambda,\quad S=[g],\quad \text{and}\ \mu=\widetilde C_q(S,\cdot).
		\]
		By \eqref{M1}, we have
		\[
		\log\frac{\overline{V_q}(K)}{\overline{V_q}(S)}
		\le
		\frac1{|\mu|}\int_{\mathbb S^{n-1}}\log\frac{h_K}{h_S}\,d\mu \quad
		\text{and}\quad \log\frac{\overline{V_q}(L)}{\overline{V_q}(S)}
		\le
		\frac1{|\mu|}\int_{\mathbb S^{n-1}}\log\frac{h_L}{h_S}\,d\mu.
		\]
		Multiplying the first inequality by $1-\lambda$ and the second  by $\lambda$, we obtain
		\[
		\log\frac{\overline{V_q}(K)^{1-\lambda}\overline{V_q}(L)^\lambda}{\overline{V_q}(S)}
		\le
		\frac1{|\mu|}\int_{\mathbb S^{n-1}}\log\frac{g}{h_S}\,d\mu.
		\]
		
		We claim that the right-hand side of the above inequality is zero. Indeed, let
		\[
		\varphi=\log\frac{g}{h_S}\ge0,
		\quad
		h_s=h_Se^{s\varphi}=h_S^{1-s}g^s,
		\quad 0\le s\le1.
		\]
		Since $h_S\le h_s\le g$ and $[h_S]=[g]=S$, it follows that $[h_s]=S$ for all $s\in[0,1]$. Hence $\log\overline{V_q}([h_s])$ is constant in $s$. Applying \eqref{HLYZ45} at $s=0$ gives
		\[
		0=\frac1{\widetilde V_q(S)}
		\int_{\mathbb S^{n-1}}\varphi\,d\widetilde C_q(S,\cdot)
		=\frac1{|\mu|}\int_{\mathbb S^{n-1}}\log\frac{g}{h_S}\,d\mu.
		\]
		Therefore,
		\[
		\overline{V_q}([h_K^{1-\lambda}h_L^\lambda])=\overline{V_q}(S)
		\ge
		\overline{V_q}(K)^{1-\lambda}\overline{V_q}(L)^\lambda.
		\]
		Taking $\lambda=\frac12$, this contradicts Theorem \ref{logmain}, and completes the proof.
	\end{proof}
	
	\begin{remark}\label{remark}
		
		The bodies \(K\) and \(L\) in Theorem \ref{min} can be chosen from
		\(\mathcal K^n_e\cap C^\infty_+\). Thus there exists a smooth, positive, even function $f$ on $\mathbb{S}^{n-1}$ such that the following Monge-Amp\`ere equation 
		\[
		h(u) |\nabla h(u)|^{q-n} \det(h_{ij}(u) + h(u) \delta_{ij}) =f
		\]
		has \emph{different} classical solutions in  $C^\infty_+$.
		
		Indeed, it suffices to derive a contradiction from the assumption that there exists \(q>n\) such that whenever \(K,L\in\mathcal K^n_e\cap C^\infty_+\) satisfy
		\(
		\widetilde C_q(K,\cdot)=\widetilde C_q(L,\cdot),
		\)
		one has $\Phi_{\widetilde{C}_q(K, \cdot)}(K)= \Phi_{\widetilde{C}_q(K, \cdot)}(L).$
		
		Let \(S\in\mathcal K_e^n\cap C^\infty_+\) and \(\mu=\widetilde C_q(S,\cdot)\). Then \(\mu\) has a smooth, positive, even density on \(\mathbb S^{n-1}\). By Lemma~\ref{smooth-existence}, there exists \(Q_0\in\mathcal K_e^n\) which is a maximizer of \(\Phi_\mu\) on \(\{Q:\widetilde V_q(Q)=|\mu|,\ Q\in\mathcal K_e^n\}\). By Lemma~\ref{HLYZ-5.1},
		$
		\widetilde C_q(Q_0,\cdot)=\mu.$
		Theorem~\ref{BF-regularity} gives \(Q_0\in\mathcal K_e^n\cap C^\infty_+\). The assumption on  uniqueness of the value of \(\Phi_\mu\) in \(\mathcal K_e^n\cap C^\infty_+\) gives \(\Phi_\mu(Q_0)=\Phi_\mu(S)\). Hence the derivation of \eqref{M1} remains valid for every smooth \(S\). By smooth approximation, the weak continuity of \(\widetilde C_q\) and the continuity of $\widetilde V_q$, we again obtain \eqref{M1} for every \(S\in\mathcal K_e^n\). Step 2 in the proof of Theorem \ref{min} then gives \eqref{logBMq}, contradicting Theorem \ref{logmain}.	
	\end{remark}
	
	\vskip3pt {\bf Conflict of Interest}: We declare that we have no conflict of interest.
	
	{\bf Data Availability}: Not applicable.

\end{document}